\documentclass[11pt,a4paper]{amsart}

\usepackage[margin=3cm]{geometry}

\usepackage{amsmath,amsthm,amsfonts,amssymb,mathtools}
\usepackage{mathrsfs}
\usepackage{stmaryrd} 
\usepackage{esint}
\usepackage{latexsym}

\usepackage{graphicx}
\usepackage[dvipsnames]{xcolor}
\DeclareGraphicsExtensions{.eps,.pdf,.jpeg,.png}

\usepackage[all]{xy}
\usepackage[new]{old-arrows}
\usepackage{pgf,tikz}
\usetikzlibrary{
  decorations.pathreplacing,
  decorations.markings,
  arrows,
  arrows.meta
}

\usepackage{array}
\usepackage{multirow}
\usepackage{blkarray,bigstrut}
\usepackage{nicematrix}
\usepackage{ytableau}

\usepackage{textcomp}
\usepackage{fancyhdr}
\usepackage{comment}
\usepackage{paralist}
\usepackage{enumerate}
\usepackage{enumitem}
\usepackage{dirtytalk}
\usepackage{url}

\usepackage[colorlinks=true,citecolor=cyan,backref=page]{hyperref}

\newcommand{\midarrow}{\tikz \draw[-triangle 90] (0,0) -- +(.1,0);}

\makeatletter
\renewcommand*\env@matrix[1][*\c@MaxMatrixCols c]{%
  \hskip -\arraycolsep
  \let\@ifnextchar\new@ifnextchar
  \array{#1}}
\makeatother

\newtheoremstyle{cleanplain}
  {10pt}      
  {10pt}      
  {\itshape}  
  {}          
  {\bfseries} 
  {.}         
  { }         
  {}

\newtheoremstyle{cleandefinition}
  {10pt}       
  {10pt}       
  {\normalfont} 
  {}           
  {\bfseries}  
  {.}          
  { }          
  {}

\theoremstyle{cleanplain}
\newtheorem{theorem}{Theorem}[section]
\newtheorem{corollary}[theorem]{Corollary}
\newtheorem{proposition}[theorem]{Proposition}
\newtheorem{lemma}[theorem]{Lemma}

\theoremstyle{cleandefinition}
\newtheorem{definition}[theorem]{Definition}
\newtheorem{remark}[theorem]{Remark}
\newtheorem{example}[theorem]{Example}

\author[N. Chapelier-Laget]{Nathan~Chapelier-Laget}

\title{$J$-ascent sets in parabolic quotients of Coxeter groups}

\begin{document}

\maketitle

\begin{abstract}
Let $(W,S)$ be a Coxeter system and let $J\subseteq S$. The right $J$-ascent set
$$
A_R^J(w)=\{s\in S\mid ws\in{}^JW,\ \ell(ws)=\ell(w)+1\}
$$
is a natural refinement of the classical ascent set adapted to the
parabolic quotient ${}^JW$. We establish a local transition formula
describing the behaviour of $J$-ascent sets under right multiplication
by a simple reflection. As a consequence, for every $w\in{}^JW$ and every
$s\in A_R^J(w)$, we prove that
$$
\Bigl||A_R^J(w)|-|A_R^J(ws)|\Bigr|
\leq \max\{1,\deg(s)-1\}.
$$
\end{abstract}


\section{Introduction}

Ascent and descent sets are among the most classical combinatorial objects studied in Coxeter groups. They describe the local structure of the weak order and appear naturally in the study of Bruhat order, Hecke algebras, Schubert calculus and many other subjects.

If $(W,S)$ is a Coxeter system and $w\in W$, the right ascent set
$$
A_R(w)=\{s\in S\mid \ell(ws)=\ell(w)+1\}
$$
records the directions in which one can move upward from $w$ in the right weak order.

\medskip

 Given $J\subseteq S$, the set ${}^JW$
is the set of minimal length representatives for the left cosets of the parabolic subgroup $W_J$. It inherits a rich combinatorial structure from the weak order on $W$.  However, the ordinary ascent set is not adapted to this restricted setting: if $w\in {}^JW$ and $s\in A_R(w)$, the element $ws$ may no longer belong to ${}^JW$.
This motivates the following object.  For $w\in {}^JW$, we define the right $J$-ascent set by
\[
A_R^J(w)=\{s\in S\mid ws\in {}^JW,\ \ell(ws)=\ell(w)+1\}.
\]

Related interactions between descent sets, reflections, and
projections onto parabolic quotients have recently been studied by
Sentinelli~\cite{Sentinelli2025}. The present approach instead
considers the ascent directions that remain inside a fixed parabolic
quotient and their behaviour under right multiplication by a simple
reflection.

\subsection{Objectives of the paper}
We first establish a local transition formula for $J$-ascent sets in arbitrary Coxeter systems.
Our first central result (see Theorem \ref{J-ascent set}) states that if $w< ws$ and both $w$ and $ws$ lie in ${}^JW$, then
$$
A_R^J(ws)= [A_R^J(w),s]\sqcup \Delta_R^J(w,s),
$$
where the notation is introduced in (\ref{crochet}) and Definition \ref{def J-ascent set}, respectively. 
This formula is the parabolic analogue of the corresponding transition rule for usual ascent sets (see Proposition \ref{ascent set}).

Our second result, Theorem~\ref{corollary structural transition}, bounds the variation in the cardinality of the $J$-ascent set under right multiplication by $s$ in terms of the degree of $s$ in the Coxeter graph:
$$
\Bigl||A_R^J(w)|-|A_R^J(ws)|\Bigr|
\leq
\max\{1,\deg(s)-1\}.
$$

 In particular, Corollary~\ref{corollary degree at most two} applies
to Coxeter groups of type $A_n$ and $\widetilde A_n$, which occur in
the Kazhdan--Lusztig setting considered in
 \cite{chapelier2024asymptotic}. 

\medskip

\section{Coxeter groups and parabolic quotients}\label{section rappels}

Let $(W,S)$ be a Coxeter system, with $|S| < \infty$, with length function $\ell$.  Let $\Phi^+$ be the generalized root system attached to  $(W,S)$.  By setting $\Phi^- = -\Phi^+$ we have $\Phi = \Phi^+ \sqcup \Phi^-$.  Let $T$ be the set of reflections of $(W,S)$, that is $T = \{wsw^{-1}~|~w \in W,~s\in S\}$.  $T$ is naturally in bijection with $\Phi^+$ and the  root associated to $t \in T$ is denoted by $\alpha_t$. 

For a subset $X$ of $W$ and $y \in W$ we denote 
\begin{equation}\label{crochet}
\displaystyle\left[X,y\displaystyle\right] := \{x \in X~|~xy = yx,~x \neq y\},
\end{equation}
and for $s \in S$, we set
\begin{equation}\label{adjacency set}
\operatorname{Adj}(s)
:=
\{r\in S \mid rs\neq sr,~r \neq s\} \quad \text{and}\quad \deg(s):= |\operatorname{Adj}(s)|.
\end{equation}

\medskip

For $w \in W$,  the right inversion set of $w$ is defined by 
$$
N_R(w) = \{t \in T~|~\ell(wt) < \ell(w)\}.
$$
An alternative description is given by 
$$
N_R(w) = \{t \in T~|~w(\alpha_t) \in \Phi^-\}.
$$

\bigskip

\subsection{Ascent sets, descent sets and right weak order}
Let $w \in W$. The right descent set of $w$ is $D_R(w) = \{s \in S~|~\ell(ws) = \ell(w) - 1\}$ and the right ascent set of $w$ is $A_R(w) = \{s \in S~|~\ell(ws) = \ell(w) + 1\}$.  Similarly, the left descent set of $w$ is $D_L(w) = \{s \in S~|~\ell(sw) = \ell(w) - 1\}$ and the left ascent set of $w$ is $A_L(w) = \{s \in S~|~\ell(sw) = \ell(w) + 1\}$.

We denote by $\leq_R$ the right weak order on $W$, that is $x \leq_R y$ if a reduced expression of $y$ has $x$ as prefix. If there is no confusion we will write $x \leq y$ instead of $x \leq_R y$ (respectively $x < y$ instead of $x <_R y$). We denote by $\Omega_W^R$ the graph (i.e. the Hasse diagram) of the right weak order on $W$. If $X \subset W$, we denote by $\Omega_W^R(X)$ the induced subgraph on $X$.

\bigskip

\subsection{Parabolic quotients}\label{subsection coset}

Let $I,J \subseteq S$. We denote by
$$
{}^JW = \{w \in W~|~J \subseteq A_L(w)\}
      = \{w \in W~|~\ell(sw)>\ell(w)\ \forall s\in J\}
$$
the set of minimal length representatives of the left cosets of $W_J$.
Similarly, we denote by
$$
W^I = \{w \in W~|~I \subseteq A_R(w)\}
     = \{w \in W~|~\ell(ws)>\ell(w)\ \forall s\in I\}
$$
the set of minimal length representatives of the right cosets of $W_I$.

When $W$ is finite, the set ${}^JW$ is an interval in the right weak order. In particular, it has a unique element of maximal length, denoted ${}^Jz=w_{0,J}w_0$, where $w_0$ is the longest element of $W$ and $w_{0,J}$ is the longest element of $W_J$  (see \cite[Section 2.2]{Geck00}).

Every element $w\in W$ admits unique decompositions
$
w=w_J\,{}^Jw
$
with
$
w_J\in W_J, {}^Jw\in {}^JW, 
$
and
$
w=w^I w_I,
$
with
$
w_I\in W_I, w^I\in W^I.
$
Moreover,
\begin{equation}\label{equality para lengths}
\ell(w)=\ell(w_J)+\ell({}^Jw)=\ell(w^I)+\ell(w_I).
\end{equation}
Here ${}^Jw$ is the unique element of minimal length in the left coset $W_Jw$, and $w^I$ is the unique element of minimal length in the right coset $wW_I$.

We also set
$$
{}^JW^I := {}^JW\cap W^I.
$$

Let $w \in W$. It is well known that the following map defines a poset isomorphism (which follows for example from \cite[Prop. 2.4.4, Prop. 3.1.6]{bjorner2005combinatorics})
$$
\begin{aligned}
(W_I,\leq_R)&\longrightarrow (wW_I,\leq_R),\\
u&\longmapsto w^{I}u.
\end{aligned}
$$

Therefore, if $xW_I=yW_I$, then $x^{I}=y^{I}$, and writing $x=x^{I}x_I$ and $y=x^{I}y_I$, we have
\begin{equation}\label{equ weak para}
    x\leq_R y
\quad\Longleftrightarrow\quad
x_I\leq_R y_I.
\end{equation}

\medskip

\subsection{A stability property}
The following lemma will be used to control when an edge of the right weak order remains inside the parabolic quotient ${}^JW$. Part $(a)$ follows from Deodhar's Lemma (see \cite[Lemma 2.1.2]{Geck00}), which is based on the fact that if an element $w\in W$ has a left ascent $r$ and a right ascent $s$, then either $\ell(rws)=\ell(w)+2$ or $rws=w$. 

\begin{lemma}[{\cite[Lemma 1.13]{CHIN20}}]\label{lemma Chin20}
 Let $w, u \in W$, $s \in S$ and $I, J \subseteq S$.
\begin{itemize}
\item[(a)] If $w\in {}^JW$ we have the following three possibilities:
\begin{itemize}
\item[(i)] $s \in A_R(w)$ and $ws \in {}^JW$;
\item[(ii)] $s \in A_R(w)$ and $ws \notin {}^JW$, in which case $ws = rw$ for some $r \in J$;
\item[(iii)] $s \in D_R(w)$ and $ws \in  {}^JW$.
\end{itemize}
\item[(b)] If $w \in {}^JW\cap W^I$ and if $u \in W_I$ then $wu \in {}^J W$ if and only if $u \in {}^{I\cap w^{-1}Jw}W$.
\end{itemize}
\end{lemma}

\bigskip

The following proposition, which stems from Lemma \ref{lemma Chin20}, shall be used later.

\bigskip

\begin{proposition}\label{stability}
Let $w \in W$.  
\begin{itemize}
\item[(1)] If $w \in {}^JW^I$ and $ws \in {}^J W$ for any $s \in I$, then $wW_I \subset {}^J W$.
\item[(2)] If $w \notin {}^JW$ then for any $z \in W$ such that $w \leq z$ we have $ \left[w,  z \right] \cap  {}^JW = \emptyset$.
\item[(3)] If $w\in{}^JW$ and $s\in A_R(w)$ are such that
$ws\notin{}^JW$, then, for every $z\in W$ satisfying $ws\leq z$,
we have $[ws,z]\cap{}^JW=\emptyset$. In particular, if
$I=A_R(ws)$, then $wsW_I\cap{}^JW=\emptyset$.
\end{itemize}
\end{proposition}

\begin{proof}
(1) Let $x \in W_I$.  By Lemma \ref{lemma Chin20} we know that $wx \in {}^J W$ if and only if $x \in {}^{I\cap w^{-1}Jw}W$.  Thus, since $ws \in  {}^J W$ for any $s \in I$, it follows that $s \in  {}^{I\cap w^{-1}Jw}W$.  However,  since $s$ is of length 1 and since $s \in I$, it follows that $s \notin w^{-1}Jw$ and then $I\cap w^{-1}Jw = \emptyset$.  Therefore,  the equivalence $wx \in {}^J W$ if and only if $x \in {}^{I\cap w^{-1}Jw}W$ becomes $wx \in {}^J W$ if and only if $x \in {}^{\emptyset}W = W$,  which is true.  Hence the first point.

(2) Since $w \notin {}^JW$ there exists $x \in J$ such that $\ell(xw) = \ell(w) - 1$.  Let $u \in  \left[w,  z \right]$ with $u = wv$ and $\ell(u) = \ell(w) + \ell(v)$.  Since $x$ is a left descent of $w$ there exists a reduced expression of $w$ starting with $x$, say $w = xw'$ with $\ell(w') =\ell(w) - 1$.  Therefore, $\ell(xu) = \ell(w'v) \leq \ell(w') + \ell(v) = \ell(u) - 1$,  which shows that $x$ is also a left descent of $u$.  Thus $u \notin  {}^JW$ and the second point follows.  

(3) The first assertion follows directly from (2), applied to $ws$.
Moreover, if $I=A_R(ws)$, then $ws\in W^I$, and hence
$ws\leq wsx$ for every $x\in W_I$. The second assertion therefore
follows again from the point (2). \qedhere
\end{proof}

\bigskip

\begin{remark}
In Proposition \ref{stability} (1) we would hope that the statement could be generalized by:  If $w \in {}^J W$ and $ws \in {}^J W$ for any $s \in I$, then $wW_I \subset {}^J W$.  However this is no longer true without the assumption  $w \in {}^JW^I$.  For example in $W = A_4$ with $J = \{s_1, s_2\}$,  $I = \{s_2, s_3\}$ and $w  = s_4s_3$,  we have $ws_3$ and $ws_2$ that belong to ${}^J W$ but $ws_3s_2s_3$ does not.
\end{remark}

\bigskip

\begin{example}
Let $W=A_3$ and $J=\{s_1\}$. Take $w=s_3$ and $s=s_1$. Then
$w\in{}^JW$, whereas $ws=s_3s_1\notin{}^JW$. Moreover, if
$z=s_1s_2s_3s_2$, then $z=s_3s_1s_2s_3$, and hence $ws\leq z$.
Proposition~\ref{stability} (3) therefore gives
$[ws,z]\cap{}^JW=\emptyset$, as illustrated in the following figure.
\end{example}

\bigskip

\begin{figure}[h!]
\begin{center}
\begin{tikzpicture}

\node at (0,9) (123121) {$123121$};

\node at (2,7.5) (12321) {$12321$};
\node at (0,7.5) (23121) {$23121$};
\node at (-2,7.5) (12312) {$12312$};

\node at (4,6) (3121) {$3121$};
\node at (2,6) (2321) {$2321$};
\node at (0,6) (1232) {$1232$};
\node at (-2,6) (2312) {$2312$};
\node at (-4,6) (1231) {$1231$};

\node at (5,4.5) (321) {$321$};
\node at (3,4.5) (232) {$232$};
\node at (1,4.5) (312) {$312$};
\node at (-1,4.5) (231) {$231$};
\node at (-3,4.5) (121) {$121$};
\node at (-5,4.5) (123) {$123$};

\node at (4,3) (32) {$32$};
\node at (2,3) (31) {$31$};
\node at (0,3) (23) {$23$};
\node at (-2,3) (21) {$21$};
\node at (-4,3) (12) {$12$};

\node at (2,1.5) (3) {$3$};
\node at (0,1.5) (2) {$2$};
\node at (-2,1.5) (1) {$1$};

\node at (0,0) (0) {$e$};

 \draw[line width=0.4mm,-{Stealth[scale=0.7]}]  (12321) to (123121);
\draw[line width=0.4mm,-{Stealth[scale=0.7]}]  (23121) to (123121);
\draw[line width=0.4mm,-{Stealth[scale=0.7]}]  (12312) to  (123121);

 \draw[line width=0.4mm ,-{Stealth[scale=0.7]}]  (3121) to (12321);
 \draw[line width=0.4mm ,-{Stealth[scale=0.7]},Dandelion]  (2321) to (23121);
 \draw[line width=0.4mm ,-{Stealth[scale=0.7]}]  (1232) to (12321);
 \draw[line width=0.4mm,-{Stealth[scale=0.7]},Dandelion]  (2312) to (23121);
\draw[line width=0.4mm,-{Stealth[scale=0.7]}]  (2312) to (12312);
\draw[line width=0.4mm,-{Stealth[scale=0.7]}]  (1231) to  (12312);

 \draw[line width=0.4mm ,-{Stealth[scale=0.7]}]  (321) to (3121);
 \draw[line width=0.4mm ,-{Stealth[scale=0.7]},Dandelion]  (321) to (2321);
 \draw[line width=0.4mm ,-{Stealth[scale=0.7]},Dandelion]  (232) to (2321);
  \draw[line width=0.4mm ,-{Stealth[scale=0.7]}]  (312) to (3121);
 \draw[line width=0.4mm ,-{Stealth[scale=0.7]}]  (312) to (1232);
\draw[line width=0.4mm,-{Stealth[scale=0.7]},Dandelion]  (231) to (2312);
\draw[line width=0.4mm,-{Stealth[scale=0.7]}]  (121) to (1231);
\draw[line width=0.4mm,-{Stealth[scale=0.7]}]  (123) to  (1232);
\draw[line width=0.4mm,-{Stealth[scale=0.7]}]  (123) to  (1231);

 \draw[line width=0.4mm ,-{Stealth[scale=0.7]},Dandelion]  (32) to (321);
 \draw[line width=0.4mm ,-{Stealth[scale=0.7]},Dandelion]  (32) to (232);
 \draw[line width=0.4mm ,-{Stealth[scale=0.7]}]  (31) to (312);
 \draw[line width=0.4mm ,-{Stealth[scale=0.7]},Dandelion]  (23) to (232);
\draw[line width=0.4mm,-{Stealth[scale=0.7]},Dandelion]  (23) to (231);
\draw[line width=0.4mm,-{Stealth[scale=0.7]},Dandelion]  (21) to (231);
\draw[line width=0.4mm,-{Stealth[scale=0.7]}]  (21) to (121);
\draw[line width=0.4mm,-{Stealth[scale=0.7]}]  (12) to (121);
\draw[line width=0.4mm,-{Stealth[scale=0.7]}]  (12) to  (123);

 \draw[line width=0.4mm ,-{Stealth[scale=0.7]},Dandelion]  (3) to (32);
\draw[line width=0.4mm,-{Stealth[scale=0.7]}]  (3) to (31);
\draw[line width=0.4mm,-{Stealth[scale=0.7]},Dandelion]  (2) to (23);
\draw[line width=0.4mm,-{Stealth[scale=0.7]},Dandelion]  (2) to (21);
\draw[line width=0.4mm,-{Stealth[scale=0.7]}]  (1) to (31);
\draw[line width=0.4mm,-{Stealth[scale=0.7]}]  (1) to  (12);

\draw[line width=0.4mm,-{Stealth[scale=0.7]}]  (0) to (1);
\draw[line width=0.4mm,-{Stealth[scale=0.7]},Dandelion] (0) to (2);
\draw [line width=0.4mm,-{Stealth[scale=0.7]},Dandelion]  (0) to (3);
\end{tikzpicture}
\end{center}
\caption{Hasse diagram of the right weak order for $W=A_3$. The yellow part is the parabolic quotient ${}^JW$ with $J = \{s_1\}$.}
\end{figure}
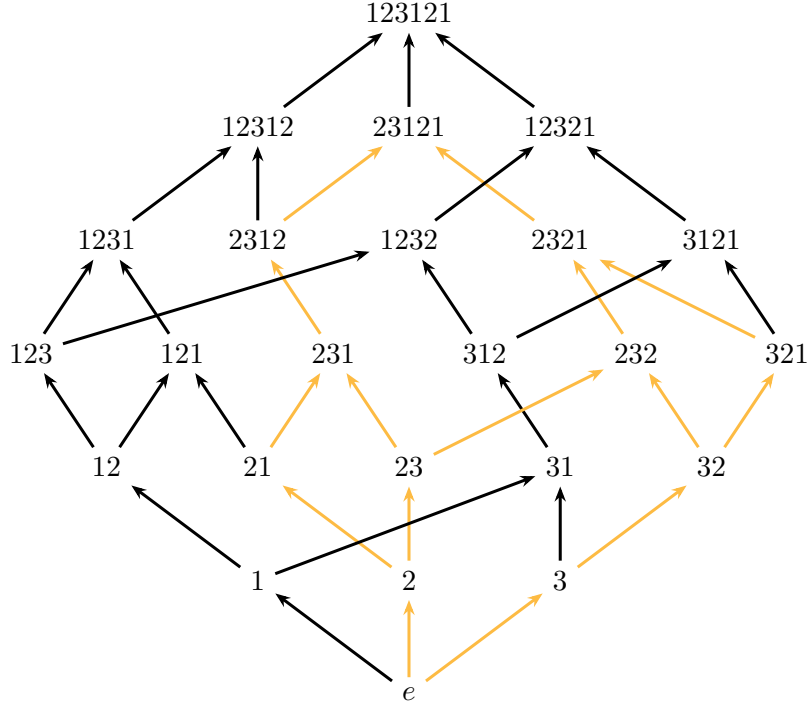

\bigskip

\section{Ascent sets}\label{section ascent sets}

Throughout this section, $(W,S)$ denotes an arbitrary Coxeter system, with $S$ finite.

\subsection{A local formula for ascent sets} In this section we study the ascent set of an element $ws$ with $w \in W$ and $s \in S$ such that $w < ws$ (for the right weak order).  The goal is to express $A_R(ws)$ in terms of $A_R(w)$. This is done in Proposition \ref{ascent set}.

\medskip

\begin{definition}
Let $w \in W$ and $s \in S$. We define $\Delta_R(w,s)$ by 
\begin{align*}
    \Delta_R(w,s) &:=  \{r \in S~|~\ell(srs)=3,~srs \notin N_R(w)\} 
                   =  \{r \in S~|~rs \neq sr,~srs \notin N_R(w)\}.
\end{align*}
In particular we see that 
\begin{equation}\label{inclusion Delta in Adj}
    \Delta_R(w,s) \subseteq \operatorname{Adj}(s).
\end{equation}

\end{definition}

\begin{proposition}\label{ascent set}
Let $w \in W$ and $s \in S$ such that $w < ws$.  We have
$$
 A_R(ws) = \displaystyle\left[ A_R(w),s \right] \sqcup \Delta_R(w,s).
 $$
\end{proposition}

\begin{proof}
Let $r \in S$ and $I = \{r,s\}$.  

\begin{itemize}
    \item Let us show the inclusion $\subset$. Assume that $r \in A_R(ws)$. 
    
    We have then $w < ws < wsr$ and it follows that $\ell(wsrs) = \ell(wsr)\pm 1$.  Thus $\ell(wsrs)$ is equal to $\ell(w) + 3$ or $\ell(w) + 1$.  This shows that $\ell(wsrs) > \ell(w)$ and then $srs \notin N_R(w)$. 

    Now, assume that $r \notin \displaystyle\left[ A_R(w),s \right]$. Therefore we have three cases: (i) $r \in A_R(w)$ and $rs \neq sr$; (ii) $r \notin A_R(w)$ and $rs = sr$; (iii) $r \notin A_R(w)$ and $rs \neq sr$. Since we already know that $srs \notin N_R(w)$, we only need to show that $rs \neq sr$, that is we must discard the case (ii). If the case (ii) stands, then we have the following configuration for the graph $\Omega_W^R(wW_I)$.

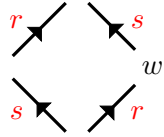
\begin{center}
    \begin{tikzpicture}
\begin{scope}[very thick, every node/.style={sloped,allow upside down}]
\node at (0,3) (9) {};

\node at (-1, 2) (7){} ;
\node at (1, 2)(6) {$w$} ;

\node at (0,1) (1) {};

\draw (7) -- node {\midarrow} (9);
\draw (6) -- node {\midarrow} (9);

\draw (1) -- node {\midarrow} (7);
\draw (1) -- node {\midarrow} (6);

\node at (0.8,2.6)  {$\textcolor{red}{s}$};
\node at (0.8,1.4)  {$\textcolor{red}{r}$};

\node at (-0.8,2.6)  {$\textcolor{red}{r}$};
\node at (-0.8,1.4)  {$\textcolor{red}{s}$};

\end{scope}
\end{tikzpicture}
\end{center}

But then, we have several ways to see a problem with this configuration: the first one being that $r$ is a descent of $ws$, a contradiction with the fact that $r \in A_R(ws)$, the second one being that $wsrs = wr < w$ and then $srs \in N_R(w)$, which is also false. 
This proves the first inclusion.

    \medskip

    \item We now show the other inclusion, for which we have two cases to consider. 
    \begin{itemize}
        \item[(i)] Assume that $m_{rs} = 2$ and let $r \in \left[ A_R(w),s \right]$. Therefore, the configuration of $\Omega_{W}^R(wW_I)$ is as follows

        
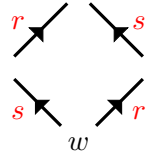
\begin{center}
    \begin{tikzpicture}
\begin{scope}[very thick, every node/.style={sloped,allow upside down}]
\node at (0,3) (9) {};

\node at (-1, 2) (7){} ;
\node at (1, 2)(6) {} ;

\node at (0,1) (1) {$w$};

\draw (7) -- node {\midarrow} (9);
\draw (6) -- node {\midarrow} (9);

\draw (1) -- node {\midarrow} (7);
\draw (1) -- node {\midarrow} (6);

\node at (0.8,2.6)  {$\textcolor{red}{s}$};
\node at (0.8,1.4)  {$\textcolor{red}{r}$};

\node at (-0.8,2.6)  {$\textcolor{red}{r}$};
\node at (-0.8,1.4)  {$\textcolor{red}{s}$};

\end{scope}
\end{tikzpicture}
\end{center}
          and this shows that $r \in A_R(ws)$, hence $\left[ A_R(w),s \right] \subseteq  A_R(ws)$.  

        \item[(ii)] Assume now that $m_{rs} \geq 3$ and let $r \in  \Delta_R(w,s)= \{r \in S~|~\ell(srs)=3,~srs \notin N_R(w)\}$.  Since $w<ws$, we have $s\in D_R(ws)$.
Therefore, if $r\notin A_R(ws)$, then $r,s\in D_R(ws)$, which
implies that $W_I$ is finite and that $ws$ is the element of maximal
length in the dihedral coset $wW_I$. But then,  $wsrs<wsr<ws$, implying that $\ell(wsrs) = \ell(ws)-2$. However, we have $\ell(ws) = \ell(w)+1$, which implies that $\ell(wsrs) = \ell(w)-1$. A contradiction since $srs \notin N_R(w)$.  This proves the inclusion $\Delta_R(w,s) \subseteq  A_R(ws)$ and then the second inclusion. \qedhere
    \end{itemize}
\end{itemize}
\end{proof}

\bigskip

\begin{remark}\label{remark local ascent comparison}
Let $w\in W$, let $s\in A_R(w)$, and let $r\in A_R(ws)$. Set
$I=\{r,s\}$ and write $w=w^Iw_I$ for the right $I$-decomposition
of $w$. The following observations are of interest.

\begin{itemize}

\item[(1)] The condition $r\in A_R(w)$ is equivalent to $w=w^I$.

\item[(2)] The inclusion $A_R(ws)\subseteq A_R(w)$ does not hold in
general. For example, in type $A_4$, if
$w=s_4s_2s_3s_1$ and $s=s_2$, then
$A_R(w)=\{s_2,s_4\}$, whereas
$A_R(ws)=\{s_1,s_3,s_4\}$.

\end{itemize}

These facts follow readily from the right $I$-parabolic decomposition
and the structure of the right weak order on the dihedral coset $wW_I$.
\end{remark}

\bigskip

\subsection{A local formula for $J$-ascent sets}
In this section we study the restriction of the right weak order to ${}^JW$.  Our main goal is to obtain, as in Proposition \ref{ascent set}, an expression of the right $J$-ascent set of $ws \in {}^JW$ in terms of the right $J$-ascent set of $w \in {}^JW$.  This is done in Theorem \ref{J-ascent set}.

\begin{definition}\label{def J-ascent set}
Let $J \subset S$ and let $w \in W$.  
\begin{itemize}
    \item[(1)] We define the right $J$-ascent set of $w$, denoted by $A_R^J(w)$,  as the set of right ascents of $w$ restricted to ${}^JW$, that is
\begin{align*}
A_R^J(w) &:= \{s \in A_R(w)~|~ws \in {}^JW\}.
\end{align*}

\item[(2)] We also define $\Delta_R^J(w,s)$ by
\begin{align*}
\Delta_R^J(w,s)&:=\{r \in S~|~wsr \in  {}^JW,~\ell(srs)=3,~srs \notin N_R(w)\} \\
						& = \Delta_R(w,s) \cap  \{r \in S~|~wsr \in {}^JW\} \\
						& =  \Delta_R(w,s) \cap  \{r \in S~|~rs \neq sr,~wsr \in {}^JW\}.
\end{align*}
The last equality above follows from the fact that the condition $rs \neq sr$ is already taken into account in $\Delta_R(w,s)$. Similarly to (\ref{inclusion Delta in Adj}), we clearly have 
\begin{equation}\label{inclusion Delta_J in Adj}
    \Delta_R^J(w,s) \subseteq \operatorname{Adj}(s).
\end{equation}

\end{itemize}
\end{definition}

\medskip

\begin{theorem}\label{J-ascent set}
Let $w \in  {}^JW$ and $s \in S$ such that $w < ws$ and $ws \in  {}^JW$.  We have
$$
 A_R^J(ws) = \displaystyle\left[ A_R^J(w),s \right] \sqcup \Delta_R^J(w,s).
 $$
\end{theorem}

\begin{proof}
Since $w<ws$, we have
$s\notin A_R(ws)$. Therefore,
$$
A_R^J(ws)
=
A_R(ws)\cap
\{r\in S\setminus\{s\}\mid wsr\in{}^JW\}.
$$

We now consider the decomposition
\begin{align*}
 \{r \in S\setminus\{s\}~|~wsr \in {}^JW\} & = \{r \in S\setminus\{s\}~|~rs = sr,~wsr \in {}^JW\} \sqcup \{r \in S\setminus\{s\}~|~rs \neq sr,~wsr \in {}^JW\} \\
 & := X_R^J(ws) \sqcup Y_R^J(ws).
\end{align*}

We claim now that 
$$
X_R^J(ws)  = \{r \in S\setminus\{s\}~|~rs=sr,~wr \in {}^JW\}.
$$  

Let $r \in X_R^J(ws)$ and assume that $wr \notin {}^JW$.  Let $I = \{s,r\}$, so that $W_I$ is a dihedral group of order four.  Since $w < ws$ and $s \in I$, by (\ref{equ weak para}) the element $w_I$ cannot be the maximal element of $W_I$, and therefore we have two cases to consider: $\ell(w_I) = 0$ or $\ell(w_I) = 1$.

Assume first that $\ell(w_I) = 0$.  Then we have $w < ws < wsr$ and $w < wr < wrs = wsr$.  Since $wr \notin {}^JW$,  we know by Proposition  \ref{stability} (2) that anything larger than $wr$ for the right weak order cannot be in ${}^JW$.  Thus, we must have $wrs \notin {}^JW$, which is false.  

Assume now that $\ell(w_I) = 1$.  It follows that $wr < w <  ws$ and $wr < wsr < ws$.  But then,  once again by Proposition \ref{stability} (2),  if $wr \notin {}^JW$ we cannot have $wsr \in {}^JW$, a contradiction.  Therefore,  in each case we must have $wr \in {}^JW$.  This proves the first inclusion. 

\medskip

We now show the other inclusion.  Let $r \in S \setminus \{s\}$ such that $rs = sr$ and $wr \in {}^JW$.  Let $I = \{s,r\}$. If $\ell(w_I) = 0$ then by Proposition \ref{stability} (1) we have $wW_I \subset {}^JW$ and then $wrs \in {}^JW$.  If $\ell(w_I) = 1$ and if $wsr \notin {}^JW$ then by Proposition \ref{stability} (2), as $wrs < ws$,  we have $ws \notin {}^JW$, which is false.  This proves the claim.

Moreover we have 
\begin{align*}
[A_R^J(w),s]
&=
\{r\in A_R^J(w)\setminus\{s\}\mid rs=sr\}\\
&=
\{r\in A_R(w)\setminus\{s\}\mid
rs=sr,\ wr\in{}^JW\}\\
&=
A_R(w)\cap X_R^J(ws).
\end{align*}

The proof follows from the next computation.

\begin{equation*}
\begin{split}
A_R^J(ws) & = A_R(ws) \cap  \{r \in S\setminus\{s\}~|~wsr \in {}^JW\} \\
				& = A_R(ws) \cap \Bigl( X_R^J(ws) \sqcup Y_R^J(ws) \Bigr) \\
				& = \Bigl(  \displaystyle\left[ A_R(w),s \right]  \sqcup \Delta_R(w,s) \Bigr)  \cap \Bigl( X_R^J(ws) \sqcup Y_R^J(ws) \Bigr)	 \qquad \text{by Prop. }~\ref{ascent set} \\
				& = \Bigl(  \displaystyle\left[ A_R(w),s \right] \cap X_R^J(ws) \Bigr) \sqcup  \Bigl(  \left[ A_R(w),s \right] \cap Y_R^J(ws)\Bigr) \sqcup  \Bigl( \Delta_R(w,s) \cap X_R^J(ws) \Bigr) ~\sqcup \\  
				& \quad \quad \Bigl( \Delta_R(w,s) \cap Y_R^J(ws) \Bigr)  \\
				& = \displaystyle\left[ A_R^J(w),s \right] \sqcup \Bigl(  \displaystyle\left[ A_R(w),s \right] \cap Y_R^J(ws)\Bigr) \sqcup  \Bigl( \Delta_R(w,s) \cap X_R^J(ws) \Bigr) ~\sqcup \\  
				& \quad \quad \Bigl( \Delta_R(w,s) \cap Y_R^J(ws) \Bigr) \\
				& = \displaystyle\left[ A_R^J(w),s \right] \sqcup \Bigl(\Delta_R(w,s) \cap Y_R^J(ws) \Bigr) \\
				& = \displaystyle\left[ A_R^J(w),s \right] \sqcup \Bigl(\Delta_R(w,s) \cap \{r \in S~|~rs \neq sr,~wsr \in {}^JW\} \Bigr) \\
				& = \displaystyle\left[ A_R^J(w),s \right] \sqcup \Delta_R^J(w,s). 
\end{split}
\end{equation*}
\end{proof}

\begin{theorem}\label{corollary structural transition}
Let $(W,S)$ be a Coxeter system, let $J\subseteq S$, let
$w\in{}^JW$, and let $s\in A_R^J(w)$. We have
\begin{itemize}
    \item[(1)] 
    $
A_R^J(ws)
=
\bigl(A_R^J(w)\setminus\{s\}\bigr)
\sqcup
\bigl(\Delta_R^J(w,s)\setminus A_R^J(w)\bigr).
$

\item[(2)] $
A_R^J(w)\setminus A_R^J(ws)=\{s\}.
$ 

\item[(3)] $
\Bigl|
|A_R^J(w)|-|A_R^J(ws)|
\Bigr|
\leq
\max\{1,\deg(s)-1\}.$

\end{itemize}
\end{theorem}

\begin{proof}
(1) Set
$
B:=A_R^J(w)\cap\operatorname{Adj}(s).
$
We first show that $B\subseteq\Delta_R^J(w,s)$. Let $r\in B$ and
set $I=\{r,s\}$. Since $r,s\in A_R^J(w)\subseteq A_R(w)$, we have
$w\in W^I$. Moreover, $w,wr,ws\in{}^JW$. It follows from
Proposition~\ref{stability}~(1) that $wW_I\subseteq{}^JW$, and hence
$wsr\in{}^JW$.
Since $w\in W^I$ and $rs\neq sr$, the word $srs$ is reduced. By
(\ref{equality para lengths}), we have
$
\ell(wsrs)=\ell(w)+\ell(srs)=\ell(w)+3.
$
Thus $srs\notin N_R(w)$, and therefore
$r\in\Delta_R^J(w,s)$. This proves that
$B\subseteq\Delta_R^J(w,s)$.

\medskip

Since $\Delta_R^J(w,s)\subseteq\operatorname{Adj}(s)$ by (\ref{inclusion Delta_J in Adj}), it follows that
$\Delta_R^J(w,s)\cap A_R^J(w) \subseteq B$. However, since $B\subseteq\Delta_R^J(w,s)$ and $B\subseteq A_R^J(w)$ (by definition), it follows that $B\subseteq\Delta_R^J(w,s) \cap A_R^J(w)$.
Therefore, 
$
B = \Delta_R^J(w,s) \cap A_R^J(w).
$
It follows then that
\begin{equation}\label{equa utile}
\Delta_R^J(w,s)
=
B\sqcup
\bigl(\Delta_R^J(w,s)\setminus A_R^J(w)\bigr).
\end{equation}

On the other hand we clearly have
$
A_R^J(w)=[A_R^J(w),s]\sqcup\{s\}\sqcup B,
$
and then it follows that
\begin{equation}\label{equa utile 2}
    A_R^J(w)\setminus\{s\}=[A_R^J(w),s]\sqcup B.
\end{equation}
Therefore
$$
\begin{aligned}
A_R^J(ws)
&=[A_R^J(w),s]\sqcup\Delta_R^J(w,s)\qquad \text{by Theorem}  ~\ref{J-ascent set}\\
&=[A_R^J(w),s]\sqcup B
  \sqcup\bigl(\Delta_R^J(w,s)\setminus A_R^J(w)\bigr) \qquad \text{by} ~ (\ref{equa utile})\\
&=\bigl(A_R^J(w)\setminus\{s\}\bigr)
  \sqcup\bigl(\Delta_R^J(w,s)\setminus A_R^J(w)\bigr) \qquad \text{by} ~ (\ref{equa utile 2}).
\end{aligned}
$$

\medskip

\noindent (2) The decomposition
$
A_R^J(ws)
=
\bigl(A_R^J(w)\setminus\{s\}\bigr)
\sqcup
\bigl(\Delta_R^J(w,s)\setminus A_R^J(w)\bigr)
$
shows that we have
$
A_R^J(w)\setminus\{s\}\subseteq A_R^J(ws),
$
and therefore 
$$
A_R^J(w)\setminus A_R^J(ws) \subseteq \{s\}.
$$
On the other hand, since $s\in A_R^J(w)$ it is clear that $s \notin A_R(ws)$ and a fortiori $s \notin A_R^J(ws)$. Hence $s\in A_R^J(w) \setminus A_R^J(ws)$.

\noindent (3) Using the above equality, we deduce  the following identity
\begin{align*}
    \big|A_R^J(w)\big|-\big|A_R^J(ws)\big| & = \big|A_R^J(w)\big| -\big|\bigl(A_R^J(w)\setminus\{s\}\bigr) \sqcup\bigl(\Delta_R^J(w,s)\setminus A_R^J(w)\bigr)\big| \\
  & =  \big|A_R^J(w)\big| - \big|\bigl(A_R^J(w)\setminus\{s\}\bigr) \big| - \big|\bigl(\Delta_R^J(w,s)\setminus A_R^J(w)\bigr)\big|\\
   & =  \big|A_R^J(w)\big| - \bigl(\big|A_R^J(w)\big| - 1 \bigr) - \big|\bigl(\Delta_R^J(w,s)\setminus A_R^J(w)\bigr)\big|\\
   & =1-\bigl|\Delta_R^J(w,s)\setminus A_R^J(w)\bigr|.
\end{align*}

Finally, the inclusion
$\Delta_R^J(w,s)\setminus A_R^J(w)\subseteq\operatorname{Adj}(s)$
implies
$
0\leq
|\Delta_R^J(w,s)\setminus A_R^J(w)|
\leq\deg(s).
$
Therefore,
$$
1-\deg(s)
\leq
|A_R^J(w)|-|A_R^J(ws)|
\leq1,
$$
and the absolute-value bound follows.
\end{proof}

\medskip

\bigskip

\begin{corollary}\label{corollary degree at most two}
Let $(W,S)$ be a Coxeter system, let $J\subseteq S$, let
$w\in{}^JW$, and let $s\in A_R^J(w)$. Assume that
$\deg(s)\leq 2$. Then
$$
\Bigl|
|A_R^J(w)|-|A_R^J(ws)|
\Bigr|
\leq 1.
$$
More precisely, we have
\begin{itemize}
\item[(i)] If $\operatorname{Adj}(s)=\emptyset$, then
$$
A_R^J(ws) = A_R^J(w) \setminus \{s\}.
$$

\item[(ii)] If $\operatorname{Adj}(s)=\{r\}$, then 
$$
A_R^J(ws)
=
\begin{cases}
A_R^J(w)\setminus\{s\},
&
\text{if } \Delta_R^J(w,s)\setminus A_R^J(w)=\emptyset,
\\[2mm]
\bigl(A_R^J(w)\setminus\{s\}\bigr)\sqcup\{r\}
&
\text{if } \Delta_R^J(w,s)\setminus A_R^J(w) = \{r\}.
\end{cases}
$$

\item[(iii)] If $\operatorname{Adj}(s)=\{r,t\}$, then
$$
A_R^J(ws)
=
\begin{cases}
A_R^J(w)\setminus\{s\},
&
\text{if }
\Delta_R^J(w,s)\setminus A_R^J(w)=\emptyset,
\\[2mm]
\bigl(A_R^J(w)\setminus\{s\}\bigr)\sqcup\{r\},
&
\text{if }
\Delta_R^J(w,s)\setminus A_R^J(w)=\{r\},
\\[2mm]
\bigl(A_R^J(w)\setminus\{s\}\bigr)\sqcup\{t\},
&
\text{if }
\Delta_R^J(w,s)\setminus A_R^J(w)=\{t\},
\\[2mm]
\bigl(A_R^J(w)\setminus\{s\}\bigr)\sqcup\{r,t\},
&
\text{if }
\Delta_R^J(w,s)\setminus A_R^J(w)=\{r,t\}.
\end{cases}
$$
\end{itemize}
\end{corollary}

\begin{proof}
The first assertion follows immediately from
Theorem~\ref{corollary structural transition} (3), since
$
\max\{1,\deg(s)-1\}=1
$
whenever $\deg(s)\leq 2$.

For the three cases, we use the formula given in Theorem \ref{corollary structural transition} (1), namely
$$
A_R^J(ws)
=
\bigl(A_R^J(w)\setminus\{s\}\bigr)
\sqcup
\bigl(\Delta_R^J(w,s)\setminus A_R^J(w)\bigr),
$$
and the fact that 
$$
\Delta_R^J(w,s) \subseteq \operatorname{Adj}(s).
$$
\end{proof}

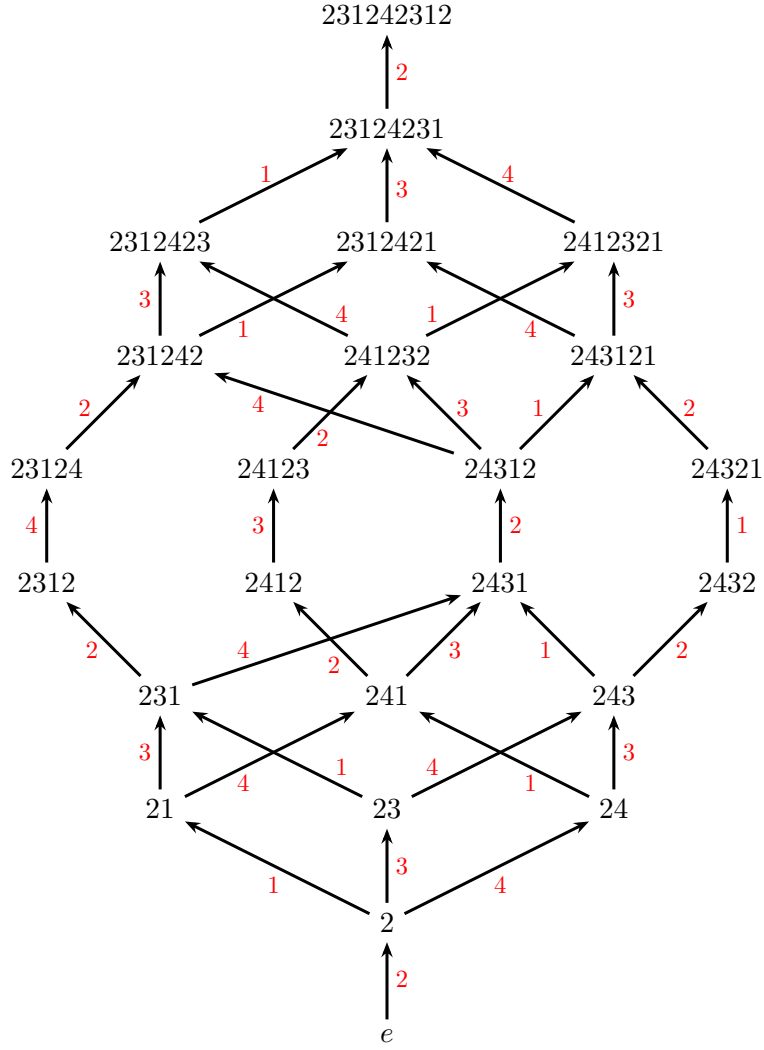
\begin{figure}[h!]
\begin{center}
\begin{tikzpicture}

\node at (0,13.5) (231242312) {$231242312$};

\node at (0,12) (23124231) {$23124231$};

\node at (3,10.5) (2412321) {$2412321$};
\node at (0,10.5) (2312421) {$2312421$};
\node at (-3,10.5) (2312423) {$2312423$};

\node at (3,9) (243121) {$243121$};
\node at (0,9) (241232) {$241232$};
\node at (-3,9) (231242) {$231242$};

\node at (4.5,7.5) (24321) {$24321$};
\node at (1.5,7.5) (24312) {$24312$};
\node at (-1.5,7.5) (24123) {$24123$};
\node at (-4.5,7.5) (23124) {$23124$};

\node at (4.5,6) (2432) {$2432$};
\node at (1.5,6) (2431) {$2431$};
\node at (-1.5,6) (2412) {$2412$};
\node at (-4.5,6) (2312) {$2312$};

\node at (3,4.5) (243) {$243$};
\node at (0,4.5) (241) {$241$};
\node at (-3,4.5) (231) {$231$};

\node at (3,3) (24) {$24$};
\node at (0,3) (23) {$23$};
\node at (-3,3) (21) {$21$};

\node at (0,1.5) (2) {$2$};

\node at (0,0) (0) {$e$};

\draw[line width=0.4mm,-{Stealth[scale=0.7]}]  (23124231) to (231242312);

\draw[line width=0.4mm,-{Stealth[scale=0.7]}]  (2412321) to (23124231);
\draw[line width=0.4mm,-{Stealth[scale=0.7]}]  (2312421) to (23124231);
\draw[line width=0.4mm,-{Stealth[scale=0.7]}]  (2312423) to (23124231);

\draw[line width=0.4mm,-{Stealth[scale=0.7]}]  (243121) to (2412321);
\draw[line width=0.4mm,-{Stealth[scale=0.7]}]  (243121) to (2312421);
\draw[line width=0.4mm,-{Stealth[scale=0.7]}]  (241232) to (2412321);
\draw[line width=0.4mm,-{Stealth[scale=0.7]}]  (241232) to (2312423);
\draw[line width=0.4mm,-{Stealth[scale=0.7]}]  (231242) to (2312421);
\draw[line width=0.4mm,-{Stealth[scale=0.7]}]  (231242) to (2312423);

\draw[line width=0.4mm,-{Stealth[scale=0.7]}]  (24321) to (243121);
\draw[line width=0.4mm,-{Stealth[scale=0.7]}]  (24312) to (243121);
\draw[line width=0.4mm,-{Stealth[scale=0.7]}]  (24312) to (241232);
\draw[line width=0.4mm,-{Stealth[scale=0.7]}]  (24312) to (231242);
\draw[line width=0.4mm,-{Stealth[scale=0.7]}]  (24123) to (241232);
\draw[line width=0.4mm,-{Stealth[scale=0.7]}]  (23124) to (231242);

\draw[line width=0.4mm,-{Stealth[scale=0.7]}]  (2432) to (24321);
\draw[line width=0.4mm,-{Stealth[scale=0.7]}]  (2431) to (24312);
\draw[line width=0.4mm,-{Stealth[scale=0.7]}]  (2412) to (24123);
\draw[line width=0.4mm,-{Stealth[scale=0.7]}]  (2312) to (23124);

\draw[line width=0.4mm,-{Stealth[scale=0.7]}]  (243) to (2432);
\draw[line width=0.4mm,-{Stealth[scale=0.7]}]  (243) to (2431);
\draw[line width=0.4mm,-{Stealth[scale=0.7]}]  (241) to (2431);
\draw[line width=0.4mm,-{Stealth[scale=0.7]}]  (241) to (2412);
\draw[line width=0.4mm,-{Stealth[scale=0.7]}]  (231) to (2431);
\draw[line width=0.4mm,-{Stealth[scale=0.7]}] (231) to (2312);

\draw[line width=0.4mm,-{Stealth[scale=0.7]}]  (24) to (243);
\draw[line width=0.4mm,-{Stealth[scale=0.7]}]  (24) to (241);
\draw[line width=0.4mm,-{Stealth[scale=0.7]}]  (23) to (243);
\draw[line width=0.4mm,-{Stealth[scale=0.7]}]  (23) to (231);
\draw[line width=0.4mm,-{Stealth[scale=0.7]}]  (21) to (241);
\draw[line width=0.4mm,-{Stealth[scale=0.7]}] (21) to (231);

\draw[line width=0.4mm,-{Stealth[scale=0.7]}]  (2) to (24);
\draw[line width=0.4mm,-{Stealth[scale=0.7]}]  (2) to (23);
\draw[line width=0.4mm,-{Stealth[scale=0.7]}] (2) to (21);

\draw [line width=0.4mm,-{Stealth[scale=0.7]}]  (0) to (2);

\node[scale=0.85] at (0.2,12.75)  {$\textcolor{red}{2}$};

\node[scale=0.85] at (1.6,11.4)  {$\textcolor{red}{4}$};
\node[scale=0.85] at (0.2,11.2)  {$\textcolor{red}{3}$};
\node[scale=0.85] at (-1.6,11.4)  {$\textcolor{red}{1}$};

\node[scale=0.85] at (3.2,9.75)  {$\textcolor{red}{3}$};
\node[scale=0.85] at (1.9,9.35)  {$\textcolor{red}{4}$};
\node[scale=0.85] at (0.6,9.55)  {$\textcolor{red}{1}$};
\node[scale=0.85] at (-0.6,9.55)  {$\textcolor{red}{4}$};
\node[scale=0.85] at (-1.9,9.35)  {$\textcolor{red}{1}$};
\node[scale=0.85] at (-3.2,9.75)  {$\textcolor{red}{3}$};

\node[scale=0.85] at (4,8.3)  {$\textcolor{red}{2}$};
\node[scale=0.85] at (2,8.3)  {$\textcolor{red}{1}$};
\node[scale=0.85] at (1,8.3)  {$\textcolor{red}{3}$};
\node[scale=0.85] at (-0.85,7.9)  {$\textcolor{red}{2}$};
\node[scale=0.85] at (-1.7,8.35)  {$\textcolor{red}{4}$};
\node[scale=0.85] at (-4,8.3)  {$\textcolor{red}{2}$};

\node[scale=0.85] at (4.7,6.75)  {$\textcolor{red}{1}$};
\node[scale=0.85] at (1.7,6.75)  {$\textcolor{red}{2}$};
\node[scale=0.85] at (-1.7,6.75)  {$\textcolor{red}{3}$};
\node[scale=0.85] at (-4.7,6.75)  {$\textcolor{red}{4}$};

\node[scale=0.85] at (3.9,5.1)  {$\textcolor{red}{2}$};
\node[scale=0.85] at (2.1,5.1)  {$\textcolor{red}{1}$};
\node[scale=0.85] at (0.9,5.1)  {$\textcolor{red}{3}$};
\node[scale=0.85] at (-0.7,4.9)  {$\textcolor{red}{2}$};
\node[scale=0.85] at (-1.9,5.1)  {$\textcolor{red}{4}$};
\node[scale=0.85] at (-3.9,5.1)  {$\textcolor{red}{2}$};

\node[scale=0.85] at (3.2,3.75)  {$\textcolor{red}{3}$};
\node[scale=0.85] at (1.9,3.35)  {$\textcolor{red}{1}$};
\node[scale=0.85] at (0.6,3.55)  {$\textcolor{red}{4}$};
\node[scale=0.85] at (-0.6,3.55)  {$\textcolor{red}{1}$};
\node[scale=0.85] at (-1.9,3.35)  {$\textcolor{red}{4}$};
\node[scale=0.85] at (-3.2,3.75)  {$\textcolor{red}{3}$};

\node[scale=0.85] at (1.5,2)  {$\textcolor{red}{4}$};
\node[scale=0.85] at (0.2,2.25)  {$\textcolor{red}{3}$};
\node[scale=0.85] at (-1.5,2)  {$\textcolor{red}{1}$};

\node[scale=0.85] at (0.2,0.75)  {$\textcolor{red}{2}$};

\end{tikzpicture}
\end{center}
\caption{Hasse diagram of the right weak order for ${}^JW$ with $W=D_4$ and $J = \{s_1,s_3,s_4\}$ where $s_2$ is the central node of the Dynkin diagram of $W$.  For the bottom edge
$e\longrightarrow s_2$, we have
$
A_R^J(e)=\{s_2\}
$
and
$
A_R^J(s_2)=\{s_1,s_3,s_4\}.
$
Thus right multiplication by $s_2$ removes $s_2$ and introduces its
three neighbours:
$
\Delta_R^J(e,s_2)\setminus A_R^J(e)=\{s_1,s_3,s_4\}.
$
In particular, the bound in
Theorem~\ref{corollary structural transition} is attained, since
$
\bigl||A_R^J(e)|-|A_R^J(s_2)|\bigr|
=\deg(s_2)-1 = 2.
$
}
\end{figure}

\bigskip

\bibliographystyle{plain}
\bibliography{J_ascent_sets_minimal.bib}

\bigskip

\small{
 \noindent \textsc{Nathan Chapelier-Laget}\\
 Université du Littoral Côte d'Opale\\
 \textit{Email address:} \texttt{nathan.chapelier@univ-littoral.fr}
}
\end{document}